\documentclass[num-refs]{wiley-networks}
\usepackage{adjustbox, hyperref} 
\usepackage{url}
\usepackage{graphicx}
\usepackage{amsmath,amssymb,amsfonts} 
\usepackage{times}
\usepackage{bbm}
\usepackage{enumitem}
\usepackage{lineno,hyperref}
\modulolinenumbers[5]
\usepackage{todonotes} 

\usepackage{courier}

\usepackage{color}

\usepackage{tikz}
\usetikzlibrary{calc,arrows,automata}
\usetikzlibrary{matrix,positioning,arrows,decorations.pathmorphing,shapes}
\usetikzlibrary{shapes,decorations}

\newtheorem{lem}{Lemma}

\newtheorem{teo}{Theorem}
\newtheorem{prop}{Proposition}

\newtheorem{cor}{Corollary}

\newcommand{\gr}[1]{{d_{#1}}}

\newcommand{\petersen}{
    \begin{tikzpicture}
      [scale=0.4,place/.style={circle,draw=black,thick,fill=black,
                 inner sep=0pt,minimum size=1mm}]

    \foreach \x in {0,1,2,3,4}{
      \node (\x) at ( 90+72*\x:1.2) [place] {};
      \node (\x') at (90+ 72*\x:2) [place] {};
    };

    \foreach \x in {0,1,2,3,4}
	  \draw (\x) to (\x');
    \foreach \x/\y in {0/2,1/3,2/4,3/0,4/1,0'/1',1'/2',2'/3',3'/4',4'/0'}
	  \draw (\x) to (\y);
	
    \end{tikzpicture}
}

\newcommand{\Kdt}{
\begin{tikzpicture}
     [ scale=0.6, nodo/.style={circle,draw=black!50,fill=black!100,
     inner sep=0pt,minimum size=1.5mm}]

\foreach \y in {1,2,3}   { 
\draw[thick] (-1, 0) to  (0,\y-2); 
\draw[thick] (1, 0) to  (0,\y-2); 
 };            
 \foreach \x in {1,2,3}     \node[nodo] (0\x) at (0,\x-2) {};
 \node[nodo] (1) at (-1,0) {};
 \node[nodo] (2) at (1,0) {};

 \end{tikzpicture} 
}

\newcommand{\Mobius}[2]{
\begin{tikzpicture}
     [ scale=0.6, nodo/.style={circle,draw=black!50,fill=black!100,
     inner sep=0pt,minimum size=3mm}]

\foreach \x in {0,...,#1}   { 
\draw[thick] (\x*#2:2) to  (\x*#2+#2:2); 
\draw[thick] (\x*#2:2) to  (\x*#2+180:2); };            
 \foreach \x in {0,...,#1}   {
  \node[nodo] (\x) at (\x*#2:2) {};}

 \end{tikzpicture} 
}

\newcommand{\cycle}[2]{
\begin{tikzpicture}
     [ scale=0.6, nodo/.style={circle,draw=black!50,fill=black!100,
     inner sep=0pt,minimum size=3mm}]

\foreach \x in {1,...,#1}   { 
\draw[thick] (\x*#2:2) to  (\x*#2+#2:2); };            
\draw[thick] (1*#2:2) to  (#1*#2+#2:2);
 \foreach \x in {1,...,#1}   {
  \node[nodo] (\x) at (\x*#2:2) {};}

 \end{tikzpicture} 
}

\newcommand{\Kunon}[2]{
\begin{tikzpicture}
     [ scale=0.6, nodo/.style={circle,draw=black!50,fill=black!100,
     inner sep=0pt,minimum size=3mm}]

  \node[nodo] (0) at (0,0) {};
\foreach \x in {1,...,#1}   { 
\draw[thick] (0) to (\x*#2:2); };            

 \foreach \x in {1,...,#1}   {
  \node[nodo] (\x) at (\x*#2:2) {};}

 \end{tikzpicture} 
}

\newcommand{\constelacion}{

\begin{tikzpicture}
     [ scale=0.55, nodo/.style={circle,draw=black!100,fill=black!100,
     inner sep=0pt,minimum size=1mm},
     nodor/.style={rectangle,draw=black!100,fill=black!100,
     inner sep=0pt,minimum size=1mm},
     nodot/.style={draw=black!100,fill=black!100,
     inner sep=0pt,minimum size=1mm}
     ]

\draw[step=1cm,gray,very thin] (1, 0) grid (20,30);

\draw[->] (1,0)--(21,0) node[right] {$n$};
\draw[->] (1,0)--(1,31)node[right] {$m$};;
\foreach \x in {1,...,20}
	\draw node[below] at (\x,0) {$\x$};
\foreach \y in {1,...,30}
	\draw node[left] at (1,\y) {$\y$};

\foreach \x in {1,...,8}{
\draw node at (\x+0.3, \x*\x*0.5-\x*0.5+0.3) {\tiny$K_{\x}$};
\draw node[nodor] at (\x, \x*\x*0.5-\x*0.5) {};
}
\foreach \x/\z in {4/2,5/2,6/3,7/3,8/4}
\foreach \y in {1,...,\z}
\draw node[nodor,green] at (\x, \x*\x*0.5-\x*0.5-\y) {};

\foreach \x in {1,...,20}{
}
\foreach \x/\y in {3/0.3333,	4/0.25,	5/0.2,	6/0.1666,	7/0.1428,	8/0.125,	9/0.1111,	10/0.1,	11/0.0909,	12/0.0833,	13/0.0769,	14/0.0714,	15/0.0666,	16/0.0625,	17/0.0588,	18/0.0555,	19/0.0526,	20/0.05}
\draw node[scale=0.15] at (\x, \x) {\cycle{\x}{360*\y}};
\foreach \x/\y in {2/.5, 3/0.3333,	4/0.25,	5/0.2,	6/0.1666,	7/0.1428,	8/0.125,	9/0.1111,	10/0.1,	11/0.0909,	12/0.0833,	13/0.0769,	14/0.0714,	15/0.0666,	16/0.0625,	17/0.0588,	18/0.0555,	19/0.0526}
\draw node[scale=0.15] at (\x+1, \x) {\Kunon{\x}{360*\y}};

\foreach \x in {5,...,19}
\draw node[scale=0.3] at (\x, \x+1) {\Kdt};
\foreach \x in {5,...,20}
\draw node[scale=0.15] at (\x, \x+2) {\Mobius{4}{90} };
\foreach \x in {6,...,20}
\draw node[scale=0.15] at (\x, \x+3) {\Mobius{6}{60}  };
\foreach \x/\y/\n in { 12/18/Y}{
\draw[left] node at (\x, \y) {\tiny$\n$};
\draw node[nodo] at (\x, \y) {};
}
\draw node[scale=0.35] at (10,15) {\petersen  };
\draw node[scale=0.15] at (8,12) {\Mobius{8}{45}  };
\foreach \r in {7,...,10}{
\draw[left] node[scale=0.8] at (2*\r, 3*\r) {$r_{\r}$};
\draw node[nodo] at (2*\r, 3*\r) {};
}
\foreach \x in {3,4}
   \draw[red] node at (2*\x, \x*2*\x-2*\x-1) {$\circ$};
\foreach \x in {3,4}
   \draw[red] node[red] at (2*\x+1, 2*\x*\x-3) {$\circ$};

\foreach \r in {6,7,8,9,10}{
\draw[thick] (2*\r, 3*\r)--(20,20+\r);
}
\foreach \r in {4}{
\draw[thick] (2*\r+0.3, 3*\r+0.3)--(20,20+\r);
}

 \end{tikzpicture}}

\newcommand{\induced}[1]{\left[#1\right]}

\papertype{Original Article}
\paperfield{Networks} 

\title{Finding uniformly most reliable graphs by counting trivial cuts}

\author[1\authfn{1}]{Eduardo Canale}
\author[2\authfn{2}]{Guillermo Rela}
\author[1\authfn{1}]{Franco Robledo}
\author[3\authfn{1}]{Pablo Romero}

\affil[1]{Laboratorio de Probabilidad y Estad\'istica, Facultad de Ingenier\'ia, Universidad de la Rep\'ublica.   Uruguay}
\affil[2]{Instituto de Ingeniería Mecánica y Producción Industrial, Universidad de la Rep\'ublica. Uruguay}
\affil[3]{Facultad de Ciencias Exactas y Naturales. Universidad de Buenos Aires. Argentina.}

\corraddress{Laboratorio de Probabilidad y Estad\'istica, Facultad de Ingenier\'ia, Universidad de la Rep\'ublica. Julio Herrera y Reissig 565, PC 11300. Montevideo, Uruguay}
\corremail{eduardo.canale@gmail.com}

\runningauthor{E. Canale et al.}

\begin{document}

\maketitle

\begin{abstract}
There is a vast literature focused on network reliability evaluation. In the last decades, 
reliability optimization has been also addressed. Frank Boesch in 1986 introduced the concept of uniformly most reliable graph (UMRG). Later, Boesch \emph{et al.} presented the first UMRGs and conjectured that some special subdivisions of the bipartite complete graph $K_{3,3}$, as well as the bipartite complete graph $K_{4,4}$, are UMRGs. Wang proved that the first conjecture is true. 
Wendy Myrvold confirmed that $K_{4,4}$ is also UMRG, by means of computational tests. However, thus far, there is no mathematical proof in the literature. A trivial cut is an edge-set that includes all the incident edges of a fixed node. In this article we describe a methodology to determine UMRGs based on bounding the number of trivial cuts. 
As a proof-of-concept it is proved that both $K_{3,3}$ and $K_{4,4}$ are UMRGs. 

\end{abstract}

\section{Motivation}\label{sect:intro}
Frank Boesch, in a seminal work~\cite{1986-Boesch}, introduced the concept of \emph{uniformly most reliable graph}, or UMRG for short. He posed several conjectures, most  of them are still awaiting for a resolution~\cite{2021-Survey}. The interested reader can find a practical discussion in the recent survey~\cite{2021-Brown}. 
In a foundational work, Boesch \emph{et al.} provided the first nontrivial UMRGs~\cite{1991-Boesch}. The authors claimed without proof that the bipartite complete graph $K_{4,4}$ is UMRG. Wendy Myrvold~\cite{1990-Myrvold} using a computational analysis found a list of all the UMRGs with 8 nodes or fewer. The list included the graph $K_{4,4}$. Nevertheless, the author included a computer-assisted proof.\\

This article is organized in the following manner. Section~\ref{def} presents the concept of UMRG as well as graph-theoretic terminology. Section~\ref{related} presents some relevant results in the field.
The main contributions are given in Sections~\ref{methodology} and~\ref{contributions}. Section~\ref{methodology} provides a novel methodology to count trivial cuts. As a proof-of-concept, 
we prove that the graphs $K_{3,3}$ and $K_{4,4}$ are UMRGs in Section~\ref{contributions}.  Section~\ref{conclusions} presents concluding remarks and trends for future work.  

\section{Definitions and Terminology}\label{def}
Throughout the document, all graphs are assumed to be simple.  
Consider a graph $G=(V,E)$ whose nodes do not fail but its edges fail independently with identical probability $\rho$. 
The all-terminal reliability $R_G(\rho)$ is the probability that the 
resulting random graph remains connected. 
For convenience, we work with the unreliability $U_G(\rho)=1-R_G(\rho)$. 
Following the terminology from~\cite{Diestel}, we denote $n=|V| =|G|$ and $m=|E|=\|G\|$ the respective order and size of the graph $G$. 
A \emph{cutset} is an edge-set $C \subseteq E$ such that the resulting
graph $G-C$ is not connected. 
Denote $m_k(G)$ the number of all the cutsets with cardinality $k$. 
By the sum-rule, the unreliability polynomial can be expressed as follows:
\begin{equation}\label{poly}
U_G(\rho) = \sum_{k=0}^{m}m_k(G) \rho^{k}(1-\rho)^{m-k}.
\end{equation}

An $(n,m)$-graph is a graph on $n$ nodes and $m$ edges.   
Observe that, for each pair of positive integers $n$ and $m$ such that $n-1 \leq m \leq \binom{n}{2}$, the collection of all $(n,m)$-graphs  is finite. Therefore, if we consider a fixed $\rho \in [0,1]$ then there exists at at least one graph $G$ that achieves the minimum unreliability, i.e., $U_G(\rho) \leq U_H(\rho)$ for all $(n,m)$-graph $H$. Further, if the previous condition holds for all $\rho \in [0,1]$ and all $(n,m)$-graphs $H$, then $G$ is a UMRG.

The following graph-theoretic terminology will be considered. 
The \emph{edge connectivity} $\lambda(G)$ is the smallest $\lambda$ such that $m_{\lambda}>0$.  
A \emph{trivial cut} is a cutset that includes all the incident edges of a fixed node. 
The \emph{degree} $\gr{v}$ of a node $v$ in $V$ is the number of edges that are incident to $v$. A graph is \emph{regular} if all its nodes have identical degrees. The \emph{minimum degree} of a graph $G$ is denoted by $\delta(G)$. 
A graph is super-$\lambda$, or \emph{superconnected}, if it is $\lambda$-regular and further, it has only trivial cutsets: $m_{\lambda} = n$. 
In a connected graph $G$, a \emph{bridge} is a single edge $uv$ such that $G-uv$ is not connected. A \emph{cut-point} is a node $v$ such that $G-\{v\}$ has more connected components than $G$. A graph $G$ with more than two nodes is \emph{biconnected} if it is connected and it has no cut-points. A \emph{tree} is an acyclic connected graph and the number of spanning trees of $G$ is its \emph{tree-number}, denoted by $t(G)$.  
A \emph{matching} is a set of nonadjacent edges. A \emph{perfect matching} is a matching that is incident to all the nodes of a graph. 
The \emph{$n$-cycle} and the \emph{$n$-complete} graphs are denoted  $C_n$ and $K_n$, respectively. The graphs $C_3$ and $C_4$ are called the \emph{triangle} and the \emph{square} respectively. The girth of a graph $G$ is denoted by $g(G)$, and it is the number of vertices in the smallest cycle belonging to $G$. 
In the \emph{bipartite complete} graph $K_{n_1,n_2}$ 
the node-set $V$ is partitioned into two parts $A$ and $B$ such that  $|A|=n_1$, $|B|=n_2$, and 
the edge-set is precisely $A \times B$.  A \emph{multipartite complete} graph $K_{n_1,\ldots,n_r}$ is defined analogously, where the node-set is $V$ is partitioned into $r$ parts $V_1, \ldots, V_r$ such that 
all the nodes belonging to $V_i$ is joined to all the nodes belonging to $V_j$, for all the pairs $i$ and $j$ such that $j\neq i$.

\clearpage

\section{Related Work}\label{related}
If $k$ is an integer such that $k\in \{0,\ldots,m\}$, then 
an $(n,m)$-graph $G$ is \emph{min-$m_k$} if $m_k(G)\leq m_k(H)$ for all the $(n,m)$-graphs $H$. Furthermore, $G$ is \emph{stronger} than $H$ if $m_k(G) \leq m_k(H)$ for all $k\in \{0,\ldots,m\}$. A graph $G$ is the \emph{strongest} in its class if it is stronger than all the graphs in its class. From Expression~\eqref{poly}, a strongest graph is UMRG. This sufficient criterion for a graph to become UMRG is widely adopted in the literature. Necessary conditions are available as well:  

\begin{prop}[\cite{2014-Brown}]\label{calculus}\
Consider two graphs $G$ and $H$ on $n$ vertices and $m$ edges. 
Then, the following assertions hold.
\begin{itemize}
\item[(i)] If there exists $k \in \{0,\ldots,m\}$ such that $m_i(H)=m_i(G)$ for all $i<k$ but $m_k(H)< m_k(G)$, 
then there exists $\rho_0 > 0$ such that $U_H(\rho) < U_G(\rho)$ for all $\rho \in (0,\rho_0)$.
\item[(ii)] If there exists $k \in \{0,\ldots,m\}$ such that $m_i(H)=m_i(G)$ for all $i>k$ but $m_k(H)< m_k(G)$, 
then there exists $\rho_1 < 1$ such that $U_H(\rho) < U_G(\rho)$ for all $\rho \in (\rho_1,1)$.
\end{itemize}
\end{prop}

Harary~\cite{1962-Harary} constructed graphs that achieve the maximum edge connectivity $\lambda = \lfloor 2m/n\rfloor$ among the class of all $(n,m)$-graphs. By definition, $m_i(G)=0$ for all $i < \lambda$. 
Then, by  Proposition~\ref{calculus}, if $G$ is UMRG then the number of cutsets $m_{\lambda}(G)$ must be minimum among the class of all $(n,m)$-graphs. On the other hand, $m_i(G)= {m \choose i}$ for all $i>m-n+1$, since trees are minimally connected with $m=n-1$ edges. 
The number of spanning subgraphs with $m-n+1$ edges is precisely the tree-number $t(G)$, so $m_{m-n+1}(G)= {m \choose m-n+1}-t(G)$. 
By Proposition~\ref{calculus}, the maximization of the tree-number is a necessary condition for a graph to become UMRG in its class. 
Prior observations directly link this network design problem with distinguished graph invariants:

\begin{cor}[Necessary Criterion] \label{necessary}
If $G$ is a UMRG then it must have the maximum tree-number $t(G)$, the maximum edge connectivity $\lambda(G)$, and the minimum number of cutsets $m_{\lambda}(G)$, among all the $(n,m)$-graphs.
\end{cor}
For convenience we say that an $(n,m)$-graph $G$ is $t$-optimal if $t(G) \geq t(H)$ for every $(n,m)$ graph $G$. Briefly, Corollary~\ref{necessary} asserts that any UMRG 
must be $t$-optimal and max-$\lambda$ min-$m_{\lambda}$.

The following theorems will be useful for our purpose:
\begin{teo}[\cite{1981-Cheng}]\label{cheng}
All regular complete multipartite graphs are $t$-optimal.  
\end{teo}

\begin{teo}[\cite{1994-Wang}]\label{stronger}
If $H$ is any $(n,m)$-graph with $m\geq n$, then there exists some stronger $(n,m)$-graph $G$ that is biconnected. 
\end{teo}

A \emph{graph constellation} with an updated set of the UMRGs found thus far is presented in Figure~\ref{constelation}. The graphical representation has the corresponding graph for every pair $(n,m)$ of nodes and edges,  whenever a UMRG exists. The pairs where UMRGs do not exist are marked with red circles~\cite{2014-Brown,1991-Myrvold}. 

The family of sparse $(n,n-i)$ graphs are straight lines with unit slope. The reader can find trees, $n$-cycles, balanced $\theta$-graphs, and some specific subdivisions of $K_4$, see~\cite{1991-Boesch}. The green squares represent dense graphs. Observe that 3-regular graphs can be found in the straight line with slope $3/2$. These graphs include $K_4$, Wagner~\cite{2017-Romero}, Petersen~\cite{2018-Rela} and Yutsis~\cite{2019-Canale}. Ath and Sobel~\cite{2000-Ath} conjectured that special subdivisions of Wagner, Petersen, Yutsis, Heawood and Cantor-Mobius are UMRGs. 
It is still an open problem to determine even if Heawood and Cantor-M\"obius are UMRGs. Thus far, the only 4-regular graphs include $K_5$, $\overline{C_3\cup C_4}$, and $K_{4,4}$ (see Theorem~\ref{Wag}).

\begin{figure}
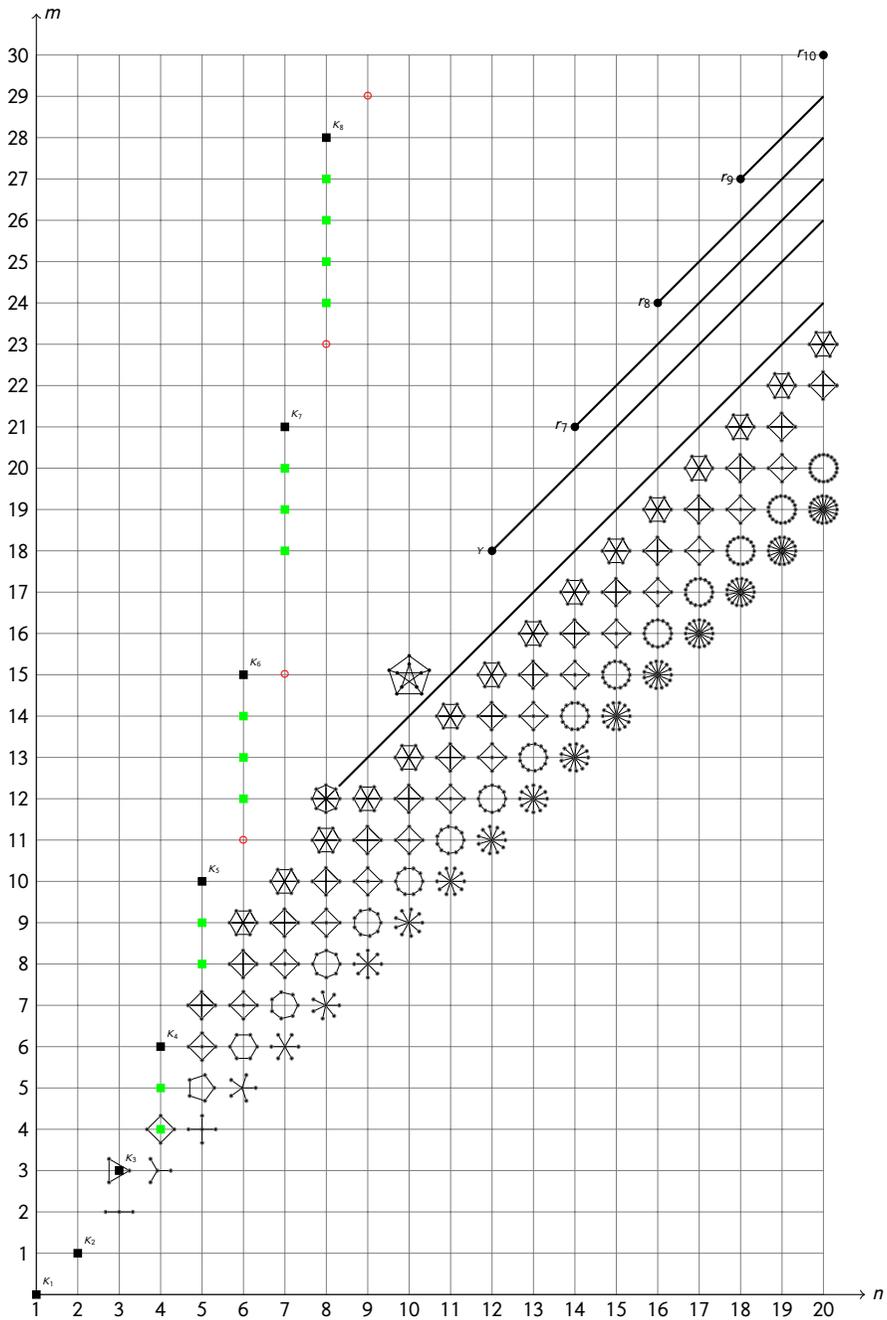
\centering{
\constelacion
}\caption{UMRGs as a function of $(n,m)$} 
	\label{constelation}
	
\end{figure} 

\clearpage

\section{Bounding Trivial Cuts}\label{methodology}
We give two upper-bounds for the coefficients $m_k(G)$ using trivial cuts. These bounds combine the inclusion-exclusion principle with elementary counting. Consider the following terminology: 
\begin{itemize}
\item For any node-set $A \subset VG$, we denote $[A]$ the subgraph induced by $A$ in $G$.
\item The cut induced by the set $A$ is $\partial A =  \{uv \in EG: u \in A, v \not\in A \}$.
\item For any set $A$, we denote $A^{(k)} = \{S \subset A: |S|=k\}$. 
\item The node-set composed by all nodes with degree $i$ is 
$V_i(G) =\{v\in VG: \gr{v}=i\}$. 
\item For any fixed node $v$, $\partial v$ denotes the set of all  edges incident to $v$,  that is, $\partial v = \partial \{v\} = \{ uv: uv \in E(G)\}$.
\item For any edge $e=uv$, $\partial e$ denotes the set of edges incident to $e$, that is, 
$\partial e = \partial \{u,v\} = \partial u \triangle \partial v$. 
\item Let $M^k(G)$ be the family of $k$-cutsets, 
$M^k(G)  = \{  S \in (EG)^{(k)}: G-S \text{ disconnected} \}$.
\item Clearly, $m_k(G) = |M^k(G)|$.
\item When the context is clear, we write $V_i$ and $M^k$ instead of $V_i(G)$ and $M^k(G)$, respectively.
\item If $H$ is a connected subgraph of $G$, $M^k_H(G)$ denotes all the $k$-cutsets  containing $\partial H=\partial VH$ but no edge belonging to $H$, i.e.  
$M^k_H(G) = \{S \in M^k(G): \partial H \subset S \subset EG\setminus EH\}$. 
\item If $H$ has only one vertex $v$ we write $M^k_v(G)$. 
\item Similarly, if $H$ has only two vertices $v$ and $w$ we will write $M^k_{vw}(G)$. 
\item For each $k \in \{1,\ldots,m\}$ and $i\leq k$
define the function $g_k$ as
$$
g_k(i) = \binom{\|G\|-i}{k-i}.$$
\end{itemize}

\begin{lem}\label{cut_bounds} If $G$ is a graph and $A$ is a subset of of $VG$ then
$$
|\partial A| = \sum_{v\in A} \gr{v} - 2\|[A]\|.
$$
Furthermore, if $G$ is 4-regular with girth $g$ and $S$ is a $k$-cutset such that $\partial A \subset S$ and that $|A|< g$, then $|A| \leq (k-2)/2$. If $|A|=g$ then $k\geq 2g$.
\end{lem}
\begin{proof}
The first part of the statement follows from the handshaking lemma and the fact that each edge in $[A]$ counts twice. Now, if $G$ is 4-regular then
$
|\partial A| = 4|A| - 2\|[A]\|.
$
If $|A|< g$ then $\|A\| \leq |A|-1$, so
 $|\partial A| \geq  4|A|-2(|A|-1)$ and $|A|\leq (|\partial A|-2)/2$. 
Since   $\partial A \subset S $, then $|\partial A| \leq k$ and the result follows. The proof for the case where $|A|=g$ is analogous.
\end{proof}

\begin{lem}\label{monotonia}
The function $g_k(i)$ is decreasing, i.e., 
$g_k( i ) \leq g_k( i-1)$ whenever $i\leq k$. 
\end{lem}
\begin{proof}
Recall that the identity $\binom{n}{i}=\frac{n}{i}\binom{n-1}{i-1}$ holds for any pair of positive integers $n$ and $i$. Then:
\begin{equation*}
g_{k}(i-1)= \binom{\|G\|-i+1}{k-i+1} = \frac{\|G\|-i+1}{k-i+1}g_k(i).
\end{equation*}
\end{proof}

\begin{prop}\label{equalities}
Consider a graph $G=(V,E)$.
\begin{itemize}
\item  If $v \in VG$ and $uv \in EG$, then 
 $\partial v \in M^{\gr{v}}$  and $\partial {uv} \in M^{\gr{v}+\gr{v}-1}$. Furthermore,  
\begin{equation}\label{contar}
|M^k_v| = \binom{\|G\| - \gr{v}}{k - \gr{v}} = g_k(d_v),\quad
\text{and }\quad  |M^k_{uv}| = \binom{\|G\| - \gr{u}-\gr{v}+1}{k  - \gr{u}-\gr{v}+2}.
\end{equation}
\item Given a subset $S$ of $VG$,  
 \begin{equation}\label{bounds}
\left|\bigcap_{v \in S } M^k_v \right| =
			g_k\left( d_{v_1}+\cdots+d_{v_i}-\|\induced{ S}\| \right) 
			\qquad \text{ where } S =\{v_1,\dots,v_i\}.
			\end{equation}
\item In particular, if $u,v \in VG$ then 
\begin{equation}\label{MkucapMkv}
		|M^k_u \cap M^k_v| = 
		\begin{cases}
		     g_k( \gr{u}+\gr{v}) &  uv \not\in E,\\
		     g_k( \gr{u}+\gr{v}-1) &  uv \in E.		     
		\end{cases}
\end{equation}
\end{itemize}
\end{prop}
\begin{proof}
The set $M^k_v$ includes precisely $d_v$ edges incident to the fixed node 
$v$ and $k-d_v$ additional edges. There are $g_k(d_v)$ ways to choose those edges, and the first equality for $|M_{v}^{k}|$ follows. The set $M_{uv}^{k}$ 
includes precisely $d_u+d_v-1$ edges incident to the fixed edge $uv$ 
and $k-d_u-d_v+2$ additional edges since $uv$ cannot be chosen, thus proving Expression~\eqref{contar}. 
Expression~\eqref{bounds} is proved analogously, and~\eqref{MkucapMkv} is a particular case of~\eqref{bounds}. 
\end{proof}

For each graph $G=(V,E)$ and each subset $A$ of $VG$, we define the functions $\overline  g_k(A)$ and $\underline{g_k}(A)$  as follows,
$$
\overline g_k (A) = 
\sum_{v \in A} g_k(\gr{v}) 
                     - \sum_{\{u,v\} \in A^{(2)}} g_k( \gr{u}+\gr{v}-\|[\{u,v\}]\|)+
\cdots +(-1)^i \sum_{S \in A^{(i)}} \!\!\!g_k\left( \sum_{v \in S}\gr{v}-\|[S]\|\right) 
                     +\cdots.
$$

$$
\underline{g_k} (A) = 
\sum_{v \in A} g_k(\gr{v}) 
                     - \sum_{\{u,v\} \in A^{(2)}} g_k( \gr{u}+\gr{v}-1)+
\cdots +(-1)^i \sum_{S \in A^{(i)}} \!\!\!g_k\left( \sum_{v \in S}\gr{v}-\frac{(-1)^i+1}2|S^{(2)}|\right)                      +\cdots.
$$
Observe that $\|[\{u,v\}]\|$ equals $1$ if and only if $uv$ is an edge and $0$ otherwise.
\begin{lem}\label{incl-excl-prop}
For any subset $A \subset VG$ and any integer $k$ such that $k\in \{0,\ldots,EG\}$:
 \begin{equation}\label{incl-excl}
 \left|\bigcup_{v \in A} M^k_v \right| 
         =  \sum_{v \in A} |M^k_v| 
                     - \sum_{u,v \in A, u\neq v} |M^k_u\cap M^k_v| + 
                      \sum_{\{u,v,w\} \in A^{(3)}} |M^k_u\cap M^k_v\cap M^k_w| -\cdots 
=  \overline g_k(A).              
\end{equation}
and  the  second summation  is 
\begin{equation}\label{*}
t = \sum_{\{u,v\} \in A^{(2)}} |M^k_u\cap M^k_v|  
= \sum_{uv \in E[A] } g_k( \gr{u}+\gr{v}-1) 
    + \sum_{\{u,v\} \in A^{(2)}\setminus E[A]} g_k( \gr{u}+\gr{v}).
\end{equation}
\end{lem}
\begin{proof}
Combine the Inclusion-Exclusion principle with \eqref{bounds} and the definition of $\overline g_k$.
\end{proof}

We are in position to prove some useful inequalities. 

\begin{lem}\label{sharpest_bound}
Let $G=(V,E)$ be a graph. If $A\subset VG$ then $m_k(G) \geq \overline g_k(A)$. 
\end{lem}
\begin{proof}
It is clear that $m_k(G)=|M^k(G)|$. Additionally, by definition,  $\cup_{v\in A}M^k_v(G) \subseteq M^k(G)$, and consequently, 
$|\cup_{v\in A}M^k_v(G)| \leq m_k(G)$. Now, the result follows from Expression~\eqref{incl-excl}.
\end{proof}

\begin{lem}\label{rude_bound}
Let $G=(V,E)$ be a graph. If $A\subseteq VG$ then $m_k(G) \geq \underline{g_k}(A)$.
\end{lem}
\begin{proof}
Combining Expression~\eqref{bounds} with the inclusion $ E[S] \subseteq S^{(2)}$ for any subset $S$ of $A$, we get that
\begin{equation}\label{base_bound}
g_k(d_{v,S} ) \leq \left|\bigcap_{v \in S } M^k_v \right| 
\leq  g_k\big(d_{v,S} -|S^{(2)}|\big),
\end{equation}
where $d_{v,S} = \sum_{v \in S}\gr{v}$. Then $\overline g_k(A) \geq  \underline{g_k}(A)$, and the statement follows by Lemma~\ref{sharpest_bound}.
\end{proof}

The following assertion constructs sharper bounds for the number of 
edge-cuts $m_k(G)$ of a graph $G$.
%

 \begin{lem}\label{SortedVertexLemma} Let $G=(V,E)$ be a graph. Consider  $A\subseteq	 VG$ and $t =\sum_{\{u,v\} \in A^{(2)}} |M^k_u\cap M^k_v|$. 
Let us sort the $|A^{(2)}|$ numbers $\{g_k(\gr{u}+\gr{v}-1)\}_{\{u,v\}\in A^{(2)}}$ increasingly and let $t'$ be its global sum. Among all those numbers consider the list all of the greatest  $h=|A^{(2)}|- \frac12\sum_{v \in A} \gr{v}$ of them, and replace those terms belonging to $t'$ by 
$g_k( \gr{u}+\gr{v})$ to obtain a new quantity $t''$. Then $m_k(G) \geq \underline{g_k}(A) +t'-t'' \geq \underline{g_k}(A)$.
 \end{lem}
\begin{proof}\,
  Since $\|[A]\| \leq \frac12\sum_{v \in A} \gr{v} $, then  
$\big|A^{(2)}\setminus E[A]| \geq h=|A^{(2)}|- \frac12\sum_{v \in A} \gr{v} $.
Therefore, there are at least $h$ non adjacent vertices in $A$ and $t$ has at least $h$ terms of the form $g_k( \gr{u}+\gr{v})$. Keeping the $h$ greatest ones, we get that  $t' \geq t'' \geq t$. 
Finally, since $\overline g_k(A) -\underline{g_k}(A) \geq -t+t' \geq -t''+t'$, then 
by Lemma~\ref{sharpest_bound} we conclude that 
$m_k(G) \geq \overline g_k(A) \geq \underline{g_k}(A) +t'-t'' \geq \underline{g_k}(A) $.
\end{proof}
In order to sort the numbers $g_k(\gr{u}+\gr{v}-1)$ we can alternatively sort the  numbers $(\gr{u}+\gr{v})$, since by Lemma~\ref{monotonia} the function $g_k$ is decreasing.

\section{The Bipartite Graphs $K_{3,3}$ and $K_{4,4}$ are UMRGs}\label{contributions}
As a proof-of-concept we show that both $K_{3,3}$ and $K_{4,4}$ are UMRGs. 
We consider essentially the bounding methodology that combines Lemmas~\ref{sharpest_bound}-\ref{SortedVertexLemma} from Section~\ref{methodology}. The proof that $K_{3,3}$ is UMRG is elementary. 
\begin{prop}
The graph $K_{3,3}$ is UMRG in the class of $(6,9)$-graphs. 
\end{prop}
\begin{proof}
Since $K_{3,3}$ has superconnectivity 3 we know that $m_{k}(K_{3,3})=0$ 
when $k\in \{0,1,2\}$. It is clear that all the $(6,9)$-graphs 
$G$ satisfy that $m_k(G)=\binom{9}{k}$ whenever $k\geq 5$. 
Furthermore, by Theorem~\ref{cheng} the graph $K_{3,3}$ is 
$t$-optimal, hence $m_4$ also attains its minimum in the graph $K_{3,3}$. 

Then, it is sufficient to prove that $m_3$ is also minimized in $K_{3,3}$. 
Let $G$ be any $(6,9)$-graph. If $G$ is regular, then it has six trivial 
3-edge-cuts, and $m_3(G)\geq 6=m_3(K_{3,3})$. 
Otherwise, $G$ is nonregular. Since $m=9$, we get that $g_3(2)= \binom{9-2}{3-2}=7$. 
By Theorem~\ref{stronger} we can assume, without loss of generality, that $d_v \geq 2$ for all $v\in VG$. Since $G$ is nonregular and its average degree equals 3, by the handshaking lemma $G$ must have some vertex $v$ 
such that $d_v=2$. Let us apply Lemma~\ref{rude_bound} using $A=\{v\}$. 
Then, 
\begin{equation*}
m_3(G) \geq \underline{g_3}(A) = g_3(gr(u))  = g_3(2) =\binom{9-2}{3-2}=\binom{7}{1} = 7 > m_3(K_{3,3}). 
\end{equation*}
We conclude that $m_k(K_{3,3}) \leq m_k(G)$ for all $G$ in $(6,9)$ and all 
$k \in \{0,\ldots,9\}$. In particular, $K_{3,3}$ is UMRG, as we wanted to prove.
\end{proof}

We will prove that $K_{4,4}$ is the strongest in its class; in particular, 
$K_{4,4}$ is UMRG. Our proof strategy can be summarized in three steps. 
First, it is straightforward to prove that $K_{4,4}$ is min-$m_k$ for all $k$ except for $k\in \{6,7,8\}$; see Lemma~\ref{elemental}. Subsequent analysis considers $k\in \{6,7,8\}$ separately. Then, we use the fact that $K_{4,4}$ is triangle-free to prove that it is the most reliable among all the 4-regular $(8,16)$-graphs. Finally, the comparison with nonregular graphs considers repeated application of Lemmas~\ref{rude_bound} and \ref{SortedVertexLemma} as well as special classifications of degree-sequences. 

\begin{lem}\label{elemental}
The graph $K_{4,4}$ is min-$m_k$, for all $k\in \{0,\ldots,5\} \cup \{9\}$
\end{lem}
\begin{proof}
Since $K_{4,4}$ is a regular complete bipartite graph, Theorem~\ref{cheng} asserts that $K_{4,4}$ is $t$-optimal, hence it is min-$m_9$. Furthermore, $K_{4,4}$ has superconnectivity $4$, and consequently, $K_{4,4}$ is min-$m_k$ for all $k\in \{0,\ldots,4\}$. Finally, observe that by Equation \eqref{MkucapMkv},  for any graph $H$ it holds that 
\begin{equation}\label{**}
M^5_v(H) \cap M^5_w(H) =\emptyset \text{ if } v\neq w \text{ with } \gr{v}+\gr{w}\geq 7.
\end{equation}
Let $H$ be a 4-regular $(8,16)$-graph. If we set $A= V(H)$ then by  Lemma~\ref{sharpest_bound}
$$
m_5(H) \geq \overline g_5(A) =8 g_5(4) = 96.
$$
Since  $K_{4,4}$ has girth 4, by Lemma~\ref{cut_bounds}, it has only trivial 5-cuts, i.e. $ M^5 \subset \cup_v M^5_v$.
Thus  $m_5(K_{4,4}) = |M_5| \leq |\cup_v M^5_v|= \sum_v |M^5_v| =8 g_5(4)=96 $, and $m_5(K_{4,4})\leq m_5(H)$, for any  4-regular  $H$.

 If $H$ is not 4-regular, by Theorem~\ref{stronger} we can assume  that $\delta(H)\in \{2,3\}$. If $\delta(H)=2$ and $\gr{v}=2$ then by Lemma~\ref{rude_bound} setting $A = \{v\}$ we have $m_5 \geq \underline{g_5}(A) = g_5(2)= 364 > 96$.
If $\delta(H)=3$ then we subdivide the discussion into two disjoint and exhaustive cases. If there are at least two nodes 
$v$ and $w$ with degree 3 then by Lemma~\ref{rude_bound} setting $A = \{v, w\}$ we have that $m_5 \geq \underline{g_5}(A) = 2g_5(3)-g_5(3+3-1)= 155 > 96$.  Otherwise, there is precisely one node $v$ with minimum degree 3. 
By the handshaking lemma, the graph $H$ must have six nodes with degree 4 and one with degree 5, hence if $\gr{w}=\gr{u}=4$ and $A=\{u,v,w\}$, by Lemma~\ref{sharpest_bound}  and \eqref{**} we get that $m_5(H) \geq \overline g_5(A)  =g_5(3)+2g_5(4)  = 102  > 96$.
\end{proof}

In the following sections we will compare $K_{4,4}$ versus regular and 
nonregular graphs in its class $(8,16)$, respectively.

\subsection{Regular Graphs}
In this section we find lower bounds on the  coefficient $m_k$ for any regular $(8,16)$-graph. The bounds are precisely $m_k(K_{4,4})$. The 
key is Mantel theorem~\cite{1907-Mantel}, which asserts that 
the maximum size of an $n$-vertex triangle-free graph $G$ is $\lfloor n^2/4 \rfloor$, and the bound is attained if and only if $G=K_{\lfloor n/2\rfloor, \lceil n/2\rceil}$.

\begin{lem}\label{regular}
If $G$ is a regular $(8,16)$-graph, then
\begin{itemize}
\item $m_6(G) \geq 8\binom{16-4}{2}+16 = 544$.
\item $m_7(G) \geq 8\binom{16-4}{3}-16+16\binom{16-7}{1} = 1888$.
\item $m_8(G) \geq 8\binom{16-4}{4}-\left[\binom{8}{2}-16\right]-16(16-7) +16\binom{16-7}{2}+8\binom{4}{2}+3\tau(G) \binom{16-8}{2}+s(G)/2$, 
\end{itemize}
where $s(G)$ is the number of squares of $G$, $\tau(G)=1$ if and only if $G$ has at least one triangle, or $0$ otherwise. Furthermore, the equalities hold if and only if $G$ is a triangle-free graph.
\end{lem}
\begin{proof}
Any disconnected graph of order $n$ will have a connected component of  at most $\lfloor n/2 \rfloor$ vertices. If $G$ is a $(8,16)$-graph and  $S \in M^k(G)$ then $G\setminus S$ will have a connected component  $H$ of cardinality at most $4$. Denote $C^i(G) $ the set of connected subgraphs of order $i$ in $G\setminus S$. Then,  
\begin{equation*}
M^k(G) = \bigcup_{i=1}^4 \bigcup_{H \in C^i(G): |\partial H|\leq k} 
M^k_{H}(G).
\end{equation*}
Let $G$ be a 4-regular $(8,16)$-graph. Consider some $H \in C^i(G)$ 
with $i\leq 4$ and $\partial H \leq 8$. There are 4 possibilities:
\begin{itemize}
\item  $H\in C^1$ and $|\partial H|$ is $4$, for which there are precisely $|G|$ possible $H$'s.
\item  $H\in C^2$ and $|\partial H|$ is $6$, for which there are precisely $\|G\|$ possible $H$'s.
\item  $H\in C^3$ and $|\partial H|$ is $8$ or $6$ depending on $[H]$ being a 3-path or a triangle. In the former, there are precisely $\binom{4}{2}|G|=6|G|$  possibilities, while in the latter the number of possibilities is a function  of the number of triangles, that is zero if and only 
if $G= K_{4,4}$.
\item $H \in C^4$ and $|\partial H|$ is  8, 6 or 4. If $|\partial H| = 8$  then $H$ is either a square or a triangle with a pendant vertex. If $|\partial H|$ is $6$ or $4$ then $H$ is either a complete graph minus an edge or the complete graph, respectively. 

\end{itemize}
Therefore,  we have the following inclusions, where   $M^k_{i} = \cup_{H\in C^i} M^k_{H}$ and $M^k_{i,j} = \cup_{H\in C^i \land |\partial H|=j} M^k_H$: 
\begin{align}
M^6 &\supseteq M^6_1 \cup M^6_2,\label{M6}\\
M^7 &\supseteq M^7_1 \cup M^7_2,\label{M7}\\
M^8 &\supseteq M^8_1 \cup M^8_2  \cup M^8_{3,8}\cup M^8_{3,6} \cup M^8_{4,8}.\label{M8}
\end{align}
These inclusions are equalities in triangle-free graphs. 
Let us consider the three cases separately:
\begin{itemize}
\item Note that $M^6_1 = \{\partial v\cup\{e,f\}: v\in V, e,f \in E\setminus \partial v\}$, $M^6_2 = \{\partial e: e \in E\}$ and 
$M_1^6\cap M_2^6 = M^6_u\cap M^6_v = M^6_e\cap M^6_f= \emptyset$, for any 
pair of vertices $u \neq v$ and edges $e \neq f$. 
The lower bound for $m_6$ then follows from the Expression \eqref{M6}.

\item Note that 
$M^7_1 = \{\partial v \cup S: v\in V, S \in (E\setminus \partial v)^{(3)}\}$, 
$M^7_2 = \{\partial e\cup\{f\}: e \in E, f\not\in \partial e \cup \{e\}\}$ and 
both $M_1^7\cap M_2^7 = M^6_e\cap M^6_f = \emptyset$ for any 
pair of edges $e\neq f$, while $|M^7_u\cap M^7_v|=1$ if  $uv\in E$, or $0$ otherwise. The lower bound for $m_7$ then follows from the Expression \eqref{M7}.
\item The lower bound for $m_8$ follows from the Expression \eqref{M8}, and  it is more involved. Clearly, 
\begin{align*}
M^8_1 &= \{\partial v\cup S: v\in V \land S \in (E \setminus \partial v)^{(4)}\}, \\
M^8_2 &= \{\partial e\cup\{f, g \}: e \in E\land f, g \in E \setminus \partial e\},\\
M^8_{3,8} &= \{\partial\{u, v, w\}: uv, vw \in E\land uw \not\in E\},\\
M^8_{3,6} &= \{\partial T\cup\{e,f\}: [T] \simeq C_3\land e,f \not\in \partial T \land |\{e,f\}\cap E[T]|\leq 1\},\\
M^8_{4,8} &= \{\partial S: [S] \simeq C_4 \}.
\end{align*}
These  sets are pairwise disjoint, so 
$$
|M^8| \geq  |M^8_1|+| M^8_2|+| M^8_{3,8}|+| M^8_{3,6}|+| M^8_{4,8}|.
$$
Let us bound each term on the right-hand side. 
For $| M^8_{4,8}|$, if $S$ is a square then in a regular $(8,16)$-graph  the complement of the square has four edges as well. That complement is either a square or a triangle with a pendant  edge. Further, if $G$ 
is triangle-free then $S^c$ must be a square. So, for each pair of complement squares there is only one set in $M^8_{4,8}$. If $G$ has 
some triangle then each set in $M^8_{4,8}$ corresponds to one or two squares; in both cases the cardinality is at least  $s/2$.

For $M^8_{3,6} $ we just bound the cardinality by considering  only one triangle $T$ and for each edge $g$ in $T$ we choose the edges $e$ and $f$ among those edges not in $\partial T\cup ET \setminus \{g\}$. Since there are three ways to choose $g$, we have $3 \binom{|E|-|\partial T\cup ET \setminus \{g\}|}{2} =
3 \binom{16-(6+2) }{2} $ possible choices for the edges $e$ and $f$.

The cardinality $|M^8_{3,8}|$ corresponds to $|G|$ choices for $v$ and 
$\binom{4}{2}$ choices for $u$ and $w$ among the neighbours of $v$. Similarly, $|M^8_2|$ corresponds to $\|G\|$ choices for $vw$ and $\binom{\|G\|-|\partial {vw}\cup\{vw\}|}{2}$ choices for the edges edges $e$ and $f$ 
that are not incident to $v$ or $w$.

Finally, to bound $|M^8_1|$ note that $\partial v\cup S= \partial {v'}\cup S'$ if and only if either $vv' \not\in E$ with $\partial v = S' $ and $\partial {v'}=S $ or 
 $vv' \in E$ and  $ \partial v \cup S = \partial {vv'} \cup \{e\}$ with $e \in EG \setminus \partial {vv'}$. 
 There are $\|\bar G\|=\binom{8}2-16$ possible choices for $vv' \not\in E$ and $\|G\|\binom{\|G\|-|\partial {vv'}|}{1}=16(16-7)$ possible choices for $vv' \in E$ and  $ \partial v \cup S = \partial {vv'} \cup \{e\}$.
\end{itemize}
\end{proof}

\begin{prop}\label{optregular}
$K_{4,4}$ is the stronger graph among all the regular $(8,16)$-graphs. 
Furthermore, $m_6(K_{4,4})=544$, $m_7(K_{4,4})=1888$ and $m_8(K_{4,4})=4446$. 
\end{prop}
\begin{proof}
The first part of the statement follows combining Mantel Theorem with Lemmas~\ref{elemental}~and~Lemma~\ref{regular}. To prove the second part of the statement use Lemma~\ref{regular} with $\tau(K_{4,4})=0$ and  $c(K_{4,4})=\binom{4}2^2$. 
\end{proof}

\subsection{Nonregular Graphs}

We will extensively use Lemmas~\ref{sharpest_bound}, \ref{rude_bound} and \ref{SortedVertexLemma}. 
The following results consider the first $h$ nodes sorted by increasing degree. Let us  call $ V^hG$ to this set, i.e., if $VG = \{v_1,\ldots, v_n\}$ with $\gr{v_1}\leq \cdots \leq \gr{v_n}$ then   $ V^hG=\{v_1,v_2,\ldots, v_h\}$. This order is not unique, but the following proofs do not depend on the specific choice.

\begin{lem}\label{preparatory1}
Among the class of $(8,16)$-graphs $G$  the coefficients $m_6$ and $m_7$ are minimized in $K_{4,4}$.
\end{lem}
\begin{proof}
Let $G$ be an $(8,16)$--graph.
If  $\delta(G)=2$, by Lemma~\ref{rude_bound}, with $A = V^1G=\{u\}$, i.e. $\gr{u}=2$, we have $ m_k(G) \geq  g_k( 2) \geq m_k(K_{4,4}) $ for $k \in \{6, 7\}$, 
 see Table~\ref{Table1}. 
 Otherwise $\delta(G)=3$ and we  discuss according to $|V_3|$, i.e., 
the number of nodes with degree $3$ in $G$:
\begin{itemize}
\item If  $|V_3| \geq 3$ then  by Lemma~\ref{rude_bound}  with  $A=V^3G$:   
$ m_k(G) \geq 3g_k(3)-3g_k(3+3-1)  = 825$, $1980$,
for $k \in \{6,7\}$, respectively, in both cases greater than the corresponding $m_k (K_{4,4})$.
\item If  $|V_3| = 2$ then $|V_4| \geq 4$. Using Lemma~\ref{rude_bound} with 
$A =V^5G$, we have
\begin{align*}
m_k(G) &\geq 2g_k( 3) + 4g_k( 4) - g_k( 3+3-1) - 6g_k( 3+4-1) - 3g_k( 4+4-1). 
\end{align*}
Then $m_6(G) \geq 753> m_6(K_{4,4})$ and $m_7(G)\geq 1972>m_7(K_{4,4})$. 
\item If $|V_3|=1$ then the  degree-sequence is $(3, 4, 4, 4, 4, 4, 4, 5)$. 
We  apply Lemma~\ref{rude_bound} with  $A =  V^7G$, to obtain 
\begin{align*}
m_k(G) &\geq g_k( 3)+6 g_k( 4) - 6 g_k( 3+4-1) - {6 \choose 2} g_k( 4+4-1).   
\end{align*}
Then $m_6(G) \geq 676> m_6(K_{4,4})$ and $m_7(G)\geq 1960>m_7(K_{4,4})$. 
\end{itemize} 
\end{proof}
 
 \begin{table}[h]
\caption{Values of $g_k( i)$  and $m_k(K_{4,4})$ for a $(8,16)$-graph $G$ and  $k=6,7,8$ .\label{Table1}}
\begin{tabular}{lrrrrrrc}
  \headrow \thead{$k$}   & \thead{$g_k( 1)$}  &  \thead{$g_k( 2)$}   &  \thead{$g_k( 3)$} &  
  \thead{$g_k( 4)$}   &  \thead{$g_k( 5)$} & 
  \thead{$g_k( 6)$} & \thead{$m_k(K_{4,4})$}\\ 
$6$   &$3003$ &  1001     &    364     & 66 & 11      &    1  &  544    \\ 
$7$   & $5005$ &  2002     &   715    & 220 & 55      &    10   &  1888   \\ 
$8$   &$6435$ &  3003     &   1287      & 495 & 165     &  45   &  4446  \\\hline 
\end{tabular}
\end{table}
 
\begin{lem}\label{preparatory2}
Among the class of $(8,16)$-graphs $G$  the coefficient  $m_8$ is minimized in $K_{4,4}$, i.e. $m_8(G) \geq \alpha = 4446  = m_8(K_{4,4})$.
\end{lem}
\begin{proof}
Let $G$ be an $(8,16)$--graph. We split the proof into two parts according to the minimum degree of $G$. Let us  write  $g(x)$ instead of $g_8(x)$. 
\begin{itemize}
\item  If $\delta(G) = 2$, we will consider three cases as a function of $|V_2(G)|$ and  $|V_3(G)|$. 
\begin{itemize}
\item If  $|V_2(G)| \geq 2$, by  Lemma~\ref{rude_bound} with $A=V^2G $, i.e. $A=\{u,v\} \subset V_2(G)$, we have, 
$$
m_8(G) \geq 
2 g(2) - g( 2+2-1) = 2\times 3003 - 1287 > \alpha.
$$
\item If $|V_2(G)| = 1$ and $|V_3| \geq 2$,  thus by  Lemma~\ref{sharpest_bound}  with 
 $A = V^3G$, we discuss according to the set $E'= E\induced{A}$ of edges of the subgraph induced by $A=\{u,v,w\}$ with $\gr{u}=2$ and $\gr{v}=\gr{w}=3$:
\begin{align*}
&m_8(G) \geq  g(2)+2g(3)\\&\quad+ \begin{cases}
 - 2g( 2+3 )- g(3+3) +g(2+3+3)		& E' = \emptyset, \\
 - 2g( 2+3 )- g(3+3-1) +g(2+3+3-1)		&E' = \{vw\},\\
 - g( 2+3-1 )- g( 2+3)- g(3+3) +g(2+3+3-1)& E'= \{uv\} \text{ or } \{uw\},\\
 -g( 2+3-1 )-g( 2+3) - g(3+3-1 )- g(2+3+3-2)	& E'= \{uv, vw\} \text{ or } \{uw, vw\}\\
 - 2g(2+3-1)- g(3+3) +g(2+3+3-2)	& E'=\{uv,uw\},\\
 - 2g( 4)- g(3+3-1) +g(2+3+3-3)			& E'= \{uv,uw,vw\}. \\
\end{cases} 
\end{align*}
which is greater than $ 4587 > \alpha$ in the six cases.\\
\item If $|V_2(G)| = 1$ and  $|V_3| \leq 1$  we consider all possible degree sequences, which are: 
 $ (2, 3, 4, 4, 4, 4, 4, 7), $ $ (2, 3, 4, 4, 4, 4, 5, 6), $ $ (2, 3, 4, 4, 4, 5, 5, 5), $ $ (2, 4, 4, 4, 4, 4, 4, 6), $ $ (2, 4, 4, 4, 4, 4, 5, 5)$, and apply  Lemma~\ref{rude_bound} with $A =  V^5G$. For the first three sequences with the prefix $(2,3,4,4,4, \ldots)$,  we have
\begin{align*}
 m_8(G) \geq   g(2)+&g(3) + 3g(4)- g( 2+3-1 )- 3g(2+4-1) -3g(3+4-1) \\
				& -3g(4+4-1) + 3g(2+3+4)  + 3g(3+4+4)+g(4+4+4)  \\
				& - 3g(2+3+4+4-C^4_2) - g(3+4+4+4-C^4_2)\\
				& - g(2+4+4+4-C^4_2) + g(2+3+4+4+4) \geq 4595 > \alpha.
\end{align*}
The terms in the second and last line are  null, because the arguments of $g$ are all greater than 8. Finally, for the last two sequences with the prefix $(2,4,4,4,4, \ldots)$,  we apply Lemma~\ref{SortedVertexLemma}  to the same $A$ to obtain
\begin{align*}
 m_8(G) &\geq  g(2)+4g(4) -  2g(2+4-1)- 2g(2+4) - 10g(4+4-1)\\
  &- 4 g(2+3\times4-\text{\small${4 \choose 2}$})= 4505> \alpha,
\end{align*}
where the null terms such as $g(2+3+4)$ were not written.
\end{itemize}  

\item If $\delta(G) = 3$ we discuss according to  $|V_3|$ and $|V_4|$:
\begin{itemize}
\item If $|V_3|\geq 5$ then, using Lemma~\ref{rude_bound} with $A=V^5G$, 
\begin{align*}
m_8(G) &\geq 5g(3)-10g(3+3-1) - 5g(3+3+3+3-6) = 4560 > \alpha.
\end{align*}
\item If $|V_3|= 4$ and $|V_4|\geq 1 $, note that 
$\|\induced{V_3}\|\leq 4$ since $G$ is biconnected. Using Lemma~\ref{sharpest_bound} with  $A=V^5G $, 
\begin{align*}
m_8(G) &\geq 4g(3)+g(4) - 4g(3+3-1)-2g(3+3) - 4g(3+4-1)\\
&- g(3+3+3+3-4)-4g(3+3+3+4-6) = 4676 > \alpha.
\end{align*}
\item If $|V_3|= 4$ and $|V_4|=0 $ then the degree sequence is $(3, 3, 3, 3, 5, 5, 5, 5)$. By Lemma~\ref{sharpest_bound} with $A=V^5G $ and the previous observation about the biconnectivity of $G$, 
\begin{align*}
m_8(G) &\geq 4g(3)+g(5) - 4g(3+3-1)-2g(3+3) - 4g(3+5-1)\\
&-g(3+3+3+3-4)-4g(3+3+3+5-6) = 4514 > \alpha.
\end{align*}
\item If $|V_3|=3$ and $|V_4|\geq 3$, if $A=V^6G $ and  $h=\|\induced{V_3}\|$, 
then $h \in \{0,1,2,3\}$ and  by Lemma~\ref{sharpest_bound} we have
\begin{align*}
m_8(G) &\geq 3g(3)+3g(4) -hg(3+3-1) -(3-h)g(3+3)-(9-2h)g(3+4-1)\\
&-2h g(3+4)-3g(4+4-1) -3g(3+3+3+4-6)\\
&- 9g(3+3+4+4-6)-g(3+3+3+4+4+4-8) \geq 4599 > \alpha.
\end{align*}

The remaining possible degree-sequences are:
$$(3, 3, 3, 4, 4, 5, 5, 5), (3, 3, 4, 4, 4, 4, 4, 6), (3, 3, 4, 4, 4, 4, 5, 5), 
(3, 4, 4, 4, 4, 4, 4, 5).$$

If the degree-sequence is $ (3, 3, 3, 4, 4, 5, 5, 5)$, by Lemma~\ref{rude_bound} with $A=V$,
\begin{align*}
m_8(G) &\geq 3g(3)+2g(4)+3g(5)- 3g(3+3-1)-g(4+4-1)\\
&-3g(5+5-1)-6g(3+4-1)-9g(3+5-1)-6g(4+5-1)\\
&-2g(3+3+3+4-6)-3g(3+3+3+5-6)-3g(3+3+4+4-6)\\
&= 4461 > \alpha.
\end{align*}

If the degree-sequences is $(3, 3, 4, 4, 4, 4, 4, 6)$ with $V_3 = \{u,v\}$, then 
the maximum number of edges between $V_3$ and $V_4$ is $4$ if  $uv \in EG$, and  $6$ otherwise.
Besides, since $|V_4|=5$ then $\|\induced{V_4}\|\leq 8$ in order to $G$ to be biconnected. Then,  by Lemma~\ref{sharpest_bound} with  $A=V$ and $h=\|\induced{V_3}\| \in \{0,1\}$, we have 
\begin{align*}
m_8(G) &\geq  2g( 3)+5g( 4)+g( 6)-g( 3+3- h)-(6-2h)g(3+4-1)\\
&-(4+2h)g(3+4)-8g( 4+4-1)-2g(3+6-1)-{4 \choose 2}g(3+3+4+4-6).
\end{align*}
The last sum is $4663$ if $h=0$ and $4615$ if $h=1$, both greater than $\alpha$.

\item Finally, if $|V_3|\leq 2$ then, the  only possible degree-sequences,  are $(3, 3, 4, 4, 4, 4, 5, 5)$  or $(3, 4, 4, 4, 4, 4, 4, 4, 5)$, the study is more involved. 
We will lower bound the cardinality of  $M'=\bigcup_{v \in VG} M^8_v \cup \bigcup _{e \in EG} M^8_e$, applying Lemma~\ref{rude_bound} to find a bound on $|\bigcup_{v \in VG} M^8_v|$, and detailed analysis for the remaining terms. By the inclusion-exclusion principle:
\begin{align*}
|M'|
&=\left|\bigcup_{v \in VG} M^8_v\right| + \left|\bigcup _{e \in EG} M^8_e\right|
-\left|\bigcup_{v \in VG} M^8_v \cap \bigcup _{e \in EG} M^8_e\right|\\
&=\left|\bigcup_{v \in VG} M^8_v\right| + \sum_{e \in EG} |M^8_e|
- \sum _{e \in EG} \sum_{v \in VG} |M^8_v \cap M^8_e|,\\
&\geq \left|\bigcup_{v \in VG} M^8_v\right| + \sum_{e \in EG} |M^8_e|
- 18,\\
\end{align*}
where the second equality holds since  $M^8_e \cap M^8_f = \emptyset$ for all $e,f \in E$, while the last inequality holds since 
$M^8_v \cap M^8_e$ is the empty sets unless $\gr{v}=3$,  $e = uw$ with $\gr{u}=\gr{w}=4$ and $v$ is adjacent to $u$ or $w$, which are only 9 cases for $|V_3|=1$ and at most 18 cases for $|V_3|=2$. 
Now, we will first bound the first term on the right-hand side of the last inequality, and then, we will bound the second term.
For the first term, let us consider the  Lemma~\ref{rude_bound} with $A = V$. For the first sequence we have 
\begin{align*}
\left|\bigcup_{v \in VG} M^8_v\right| &\geq 
2g(3)+4g(4)+2g(5)- f(3+3-1)-6g(4+4-1) - 8g(3+4-1)\\
& -4g(3+5-1) -8g(4+5-1)	-{4 \choose 2} g(3+3+4+4-6)= 4255,
\end{align*}
and for the second sequence we have 
\begin{align*}
\left|\bigcup_{v \in VG} M^8_v\right| &\geq
     g(3)+6g(4)+g(5) - 6g(3+4-1) - g(3+5-1) - 6g(4+5-1)\\
      &= 4137.
\end{align*} 
Now, let us consider the sum  $S = \sum_{e \in EG} |M^8_e|$. If $e = uv$, then $|M^8_e| = \binom{16-\gr{u}-\gr{v}+1}{8-\gr{u}-\gr{v}}$, which has the four possible values shown in Table~\ref{Table2}, since there  are   three possible values of the  degrees. Thus, a possible lower bound for that sum $S$ is the minimum of the function $f(a,b,c,d) = 168 a + 36 b + 36 c + 8 d$ subject to the constraints $a+b =5, 6$, $d =9, 10$ and $a+b+c+d = 16$ for the first  sequence and subject to the  constraints $a+b = 3$, $d \leq 5$ and $a+b+c+d = 16$ for the second sequence.

These minimum values are attained at  $(a,b,c,d)=(0,5,0,11)$, and  
$(a,b,c,d)=(0,3,5,8)$, respectively. Therefore, 
$S\geq 268$ and $ 436 $ respectively. All the bounds together give us: 
$m_8(G) \geq 4255+268-18 = 4505$ and,  
$m_8(G) \geq 4162+436-18 = 4580 $, both greater than $\alpha$.
\end{itemize}
\end{itemize}
\end{proof}

\begin{table}
\caption{ Values of $|M^8_{uv}| $ according with $\gr{u}$ and $\gr{v}$\label{Table2}}
\begin{tabular}{lcc}
\headrow $\gr{u}$ & $\gr{v}$ &  $|M^8_{uv}|$\\ 
 3     &   4  & 168       \\ 
 3     &   5   & 36    \\  
 4     &   4    & 36   \\ 
 4     &   5    & 8  \\\hline 
\end{tabular}
\end{table}

\begin{prop}\label{optnoregular}
Among the class of nonregular $(8,16)$-graphs, $K_{4,4}$ is min-$m_k$ for all $k \in
\{0,\ldots,16\}$.
\end{prop}
\begin{proof}
By Lemmas~\ref{elemental}, \ref{preparatory1}~and~\ref{preparatory2} we know 
that $K_{4,4}$ is min-$m_k$ for each $k$ in $\{0,\ldots,9\}$. Now, set $k \in \{10,\ldots,16\}$. Observe that if we remove $k$ edges to any $(8,16)$-graph then 
the resulting subgraph is not connected. Therefore, $m_{k}(K_{4,4})=m_k(G)=\binom{16}{k}$ for each $k \in \{10,\ldots,16\}$. 
\end{proof}

\begin{teo}\label{Wag}
The complete bipartite graph $K_{4,4}$ is the strongest in its class. 
In particular, $K_{4,4}$ is uniformly most reliable in the class of $(8,16)$-graphs. 
\end{teo}
\begin{proof}
Combining Propositions~\ref{optregular}~and~\ref{optnoregular}, we conclude that 
the complete bipartite graph $K_{4,4}$ is the strongest in its class. In particular, any strongest graph in its class is uniformly most reliable, and the result follows.
\end{proof}

\section{Conclusions and Trends for Future Work}
\label{conclusions}
Uniformly most reliable graphs (UMRGs) represent a synthesis in network reliability analysis. Finding them is a hard 
task not well understood. An exhaustive comparison is computationally prohibitive for most cases. 
Prior works in the field try to globally minimize the cutsets. This minimization could be enriched with our bounding methodology. As a proof-of-concept, here we 
show that both $K_{3,3}$ and $K_{4,4}$ are UMRGs. There are several trends for future work. Even though  all UMRGs share special properties such as the greatest girth, there are no proofs available for these conjectures. 
A powerful methodology to find UMRGs is not known.  In this paper we propose a novel methodology to count trivial cutsets, which could be used for the discovery of UMRGs. As future work, we want to study if the complete bipartite graphs are always UMRGs.

\bibliography{biblio}

\begin{thebibliography}{10}

\bibitem{2000-Ath}
Y.~Ath and M.~Sobel, {\em Some conjectured uniformly optimal reliable
  networks}, Prob.   Eng.  Informat. Sciences. {\bf 14} (2000),  375–383.

\bibitem{1986-Boesch}
F.T. Boesch, {\em On unreliability polynomials and graph connectivity in
  reliable network synthesis}, J.  Graph Theory {\bf 10} (1986),  339--352.

\bibitem{1991-Boesch}
F.T. Boesch, X.~Li, and C.~Suffel, {\em On the existence of uniformly optimally
  reliable networks}, Networks {\bf 21} (1991),  181--194.

\bibitem{2021-Brown}
J.I. Brown, C.J. Colbourn, D.~Cox, C.~Graves, and L.~Mol, {\em Network
  reliability: Heading out on the highway}, Networks {\bf 77} (2021),
  146--160.

\bibitem{2014-Brown}
J.I. Brown and D.~Cox, {\em Nonexistence of optimal graphs for all terminal
  reliability}, Networks {\bf 63} (2014),  146--153.

\bibitem{2019-Canale}
E.A. Canale, F.~Robledo, P.~Romero, and J.~Viera, {\em Building
  reliability-improving network transformations}, Proceedings of the 15th
  International Conference on the Design of Reliable Communication Networks,
  {IEEE}, 2019, pp.  107--113.

\bibitem{1981-Cheng}
C.S. Cheng, {\em Maximizing the total number of spanning trees in a graph: Two
  related problems in graph theory and optimum design theory}, J.
  Combinatorial Theory, Ser. B {\bf 31} (1981),  240--248.

\bibitem{Diestel}
R.~Diestel, {\em Graph Theory}, Springer-Verlag, 2000.

\bibitem{1962-Harary}
F.~Harary, {\em The maximum connectivity of a graph}, Proc.   National Acad.
  Sci. {\bf 48} (1962),  1142--1146.

\bibitem{1907-Mantel}
W.~Mantel, {\em Problem 28 {(Solution by H. Gouwentak, W. Mantel, J. Teixeira
  de Mattes, F. Schuh and W. A. Wythoff)}}, Wiskundige Opgaven {\bf 10} (1907),
   60–61.

\bibitem{1990-Myrvold}
W.~Myrvold, {\em Uniformly-most reliable graphs do not always exist},
  Tech.~report, Department of Computer Science, University of Victoria,
  Victoria, B.C., Canada, 1990, Technical Report \#DCS-120-IR.

\bibitem{1991-Myrvold}
W.~Myrvold, K.H. Cheung, L.B. Page, and J.E. Perry, {\em Uniformly-most
  reliable networks do not always exist}, Networks {\bf 21} (1991),  417--419.

\bibitem{2018-Rela}
G.~Rela, F.~Robledo, and P.~Romero, {\em Petersen graph is uniformly
  most-reliable}, Machine Learning, Optimization, and Big Data, G. Nicosia et
  al. (eds.), Springer International Publishing, Cham, 2018, pp.  426--435.

\bibitem{2021-Survey}
P.~Romero, {\em Uniformly optimally reliable graphs: A survey}, Networks,
  https://doi.org/10.1002/net.22085 ().

\bibitem{2017-Romero}
P.~Romero, {\em Building uniformly most-reliable networks by iterative
  augmentation}, International Workshop on Resilient Networks Design and
  Modeling, Alghero, Sardinia, Italy, 2017.

\bibitem{1994-Wang}
G.~Wang, {\em {A proof of Boesch's conjecture}}, Networks {\bf 24} (1994),
  277--284.

\end{thebibliography}

\end{document}